\documentclass{amsart}

\usepackage{mathtools}

\usepackage{xcolor}

\usepackage{amsaddr}
\usepackage{amssymb}
\usepackage{amsthm}

\usepackage[margin=3cm]{geometry}

\usepackage{enumitem}
\usepackage[ruled,noline,titlenumbered,linesnumbered,noend,norelsize]{algorithm2e}
\DontPrintSemicolon
\SetNlSty{footnotesize}{}{}
\IncMargin{0.5em}
\SetNlSkip{1em}
\SetKwIF{If}{ElseIf}{Else}{if}{then}{else if}{else}{endif}
\SetKw{Return}{return}
\SetKw{Pick}{pick}
\SetNlSkip{2em}
\newenvironment{algorithm-hbox}{\hbadness=10000\begin{algorithm}}{\end{algorithm}}

\SetAlgorithmName{Listing}{Listings}

\newtheorem{theorem}{Theorem}

\newtheorem{prop}[theorem]{Proposition}
\newtheorem{observation}[theorem]{Observation}

\renewcommand{\epsilon}{\varepsilon}

\renewcommand{\leq}{\leqslant}
\renewcommand{\geq}{\geqslant}

\newcommand{\Nat}{\mathbb{N}}

\newcommand{\redset}{\mathcal{R}}
\newcommand{\blueset}{\mathcal{B}}

\newcommand{\expt}{\mathbb{E}}

\newcommand{\prob}{\mathrm{Pr}}

\newcommand{\tl}[1]{\mathcal{T}_#1}

\newcommand{\eul}{\mathrm{e}}

\newcommand{\Oh}[1]{\mathcal{O}{\left(#1\right)}}
\newcommand{\oh}[1]{o{\left(#1\right)}}
\newcommand{\Th}[1]{\Theta{\left(#1\right)}}

\newcommand{\expb}[1]{\exp{\left( {#1} \right)}}

\begin{document}
\title{Random hypergraphs and Property B}

\author{Lech Duraj} 
\email{Lech.Duraj@uj.edu.pl}

\author{Jakub Kozik} 
\address{Theoretical Computer Science Department, Faculty of Mathematics and Computer Science, Jagiellonian University, Krak\'{o}w, Poland}
\email{Jakub.Kozik@uj.edu.pl}

\author{Dmitry Shabanov}
\email{dmitry.shabanov@phystech.edu}

\address{
Moscow Institute of Physics and Technology, Laboratory of Combinatorial and Geometric Structures, Dolgoprudny, Moscow Region, Russia
\\National Research University Higher School of Economics, Faculty of Computer Science, Moscow, Russia
}

\thanks{Research of L. Duraj and J. Kozik was partially supported by Polish National Science Center (2016/21/B/ST6/02165).}

\thanks{Research of D. Shabanov was supported by grant of the President of Russian Federation no. MD-1562.2020.1 and by the program "Leading Scientific Schools", grant no. NSh- 2540.2020.1.}

\begin{abstract}
In 1964 Erd\H{o}s proved that $(1+\oh{1})) \frac{\eul \ln(2)}{4} k^2 2^{k}$ edges are sufficient to build a $k$-graph which is not two colorable.
To this day, it is not known whether there exist such $k$-graphs with smaller number of edges.
Erd\H{o}s' bound is consequence of the fact that a hypergraph with $k^2/2$ vertices and $M(k)=(1+\oh{1}) \frac{\eul \ln(2)}{4} k^2 2^{k}$ randomly chosen edges of size $k$ is asymptotically almost surely not two colorable.
Our first main result implies that for any $\varepsilon > 0$, any $k$-graph with $(1-\varepsilon) M(k)$ randomly and uniformly chosen edges is a.a.s. two colorable.
The presented proof is an adaptation of the second moment method analogous to the developments of Achlioptas and Moore from 2002 who considered the problem with fixed size of edges and number of vertices tending to infinity.
In the second part of the paper we consider the problem of algorithmic coloring of random $k$-graphs.
We show that quite simple, and somewhat greedy procedure, a.a.s. finds a proper two coloring for random $k$-graphs on $k^2/2$ vertices, with at most $\Oh{k\ln k\cdot 2^k}$ edges. 
That is of the same asymptotic order as the analogue of the \emph{algorithmic barrier} defined by Achlioptas and Coja-Oghlan in 2008, for the case of fixed $k$.
\end{abstract}

\maketitle

\section{Introduction}

The smallest number of edges in a $k$-graph (i.e. $k$-uniform hypergraph) that is not two colorable is denoted by $m(k)$.
Early results by Erd\H{o}s from 60's \cite{Erd1963,Erd1964} determined that the exponential factor of the growth of $m(k)$ is $2^k$ (that is $\log_2(m(k)) \sim k$).
In 1975 Erd\H{o}s and Lov\'{a}sz in \cite{EL1975} suggested that \emph{perhaps $k \cdot 2^{k}$ is the correct order of magnitude for $m(k)$}.

Despite a few improvements on the side of the lower bounds (the most recent one by Radhakrishnan and Srinivasan in \cite{RS00}) the upper bound of Erd\H{o}s from 1964 has not been improved since.
He proved in \cite{Erd1964} that
\begin{align}
\label{eq:ErdUB}
	m(k) \leq (1+\oh{1}) \frac{\eul \ln(2)}{4} k^2 2^{k}.
\end{align}

This bound results from the fact that the random hypergraph with that number of edges, over the set of vertices of size about $k^2/2$, a.a.s. can not be properly two colored.
The chosen number of vertices, turns out to give the smallest number of edges in the random construction.
Our motivation for taking a closer look on random $k$-graphs on $k^2/2$ vertices stems mainly from this fact.
Within these constrains we address two problems. First, we show that the bound of Erd\H{o}s is tight for random hypergraphs.
Then, we focus on random $k$-graphs with smaller number of edges and discuss the problem of finding a proper coloring efficiently.

Analogous problems have been considered in the context of random constraint satisfaction problems, where the size $k$ of constraints/edges was fixed and the number of variables/vertices $n$ was tending to infinity.
In that framework, the straightforward first moment calculation shows that if $r \geq 2^{k-1} \ln(2)- \ln(2)/2$, then $H_k(n, r \cdot n)$ is a.a.s. not two colorable.
That observation was complemented by Achlioptas and Moore \cite{AM2002} who proved that,
if $r<2^{k-1} \ln(2)- (1+\ln(2))/2+o_k(1)$, then random hypergraph $H_k(n, r \cdot n)$ is a.a.s. two colorable.
This is the first of the papers that directly inspire our developments
(the bound itself has been later improved in \cite{CP2012} and \cite{CZ2012}).

Second result from that area which provides a context for our considerations on efficient coloring, was proved by Achlioptas and Coja-Oghlan in \cite{ACO08}.
They discussed the evolution of the space of solutions when successive random constraints are added to the instance.
They observed that, at some point, the set of solutions, which in certain sense is initially connected, undergoes a sudden change after which it becomes \emph{shattered} into exponentially many well separated regions.
That behavior has been interpreted as a barrier for effective algorithms.
Since that time, for some specific problems, algorithms and their analyses were improved up to the threshold of shattering (see e.g. \cite{CojaO2010}) but none surpassed that barrier.

The case of our interests, when $n= \Theta(k^2)$, is slightly different, and most proofs for fixed $k$ do not translate literally.
The ideas of the proofs however usually can be applied. As shown by our work, the resulting proofs turn out to be technically simpler.
Moreover, since the values like $2^{k-1}$ are no longer constants, we were able to obtain the sharp threshold for two colorability of the random hypergraph in many cases. Note that for fixed $k$, this question is still open (however the results in \cite{CP2012} give very tight bounds).
For a similar problem $k$-SAT with not too slowly growing $k$, Frieze and Wormald in \cite{FW05} proved the existence of satisfiability threshold by analogous methods.

\section{Main results}
When discussing two coloring of a hypergraph, we use colors blue and red. 
We analyze the problem of coloring of random $k$-graphs with large $k$.
Therefore asymptotic statements shall be understood with $k\rightarrow \infty$. 
We say that a property holds \emph{asymptotically almost surely} if the probability that it holds is $1-\oh{1}$.
For positive integer $b$ we denote by $a^{\underline{b}}$ the falling factorial, i.e. 
\[
	a^{\underline{b}} = \prod_{j=0,\ldots,b-1} (a-j).
\]
For positive integers $n\geq 2k$, we define $\varphi= \varphi(n,k)$ to be such that
\[
	\frac{(n/2)^{\underline{k}}}{n^{\underline{k}}} = \frac{\varphi(n,k)}{2^{k}}.
\]
Note that always $2^{-k} \leq \varphi(n,k)\leq 1$.
Moreover, for $n= \omega( k^{3/2})$, we have $\varphi(n,k) \sim \exp(-k^2/(2n))$. 
In particular, $ n = \omega( k^2)$ implies $\varphi(n,k)\sim 1$.

The main subjects of our considerations are random $k$-graphs.
We denote the number of vertices by $n$, and the number of distinct uniformly sampled edges by $m$.
The resulting probabilistic space is denoted by $H_k(n,m)$.
In other words $H_k(n,m)$ can be seen as uniformly chosen $k$-graph on $n$ vertices with $m$ edges.
For technical convenience we assume that $n$ is always even.
Our main interest lies in $k$-graphs with $\Theta(k^2)$ vertices, but the presented proof covers much larger range of parameters.

\subsection{Sharp threshold for two colorability}
\begin{theorem}
\label{thm:main}
	For any $\varepsilon>0$ and any superlinear function $n=n(k)$ satisfying $\ln(n) = \oh{k}$ 
	if 
	\[
		c= c(n,k) < (1-\varepsilon) \ln(2)/\varphi(n,k),
	\]
	then hypergraph $H_k(n,c n 2^{k-1})$ is a.a.s. two colorable (when $k$ goes to infinity).
\end{theorem}
Observe that for $\varepsilon = 0$, the number of edges $c n 2^{k-1}$ is asymptotically equivalent to the upper bound (\ref{eq:ErdUB}).
Note also that, by using the first moment argument, it is straightforward to check that if
\[
		c= c(n,k) > (1+\varepsilon) \ln(2)/\varphi(n,k),
\]
then hypergraph $H_k(n,c n 2^{k-1})$ is not two colorable a.a.s., which means that Theorem 1 provides the sharp threshold for the two colorability of $H_k(n,m)$ for a growing function $k$ satisfying $\ln n=\oh{k}$ and superlinear number of vertices.

In the proof of Theorem \ref{thm:main} we apply the second moment method along the lines of developments of Achlioptas and Moore from \cite{AM2002}.
An important modification that allowed to essentially simplify the argument is that we focus only on equitable colorings (following the idea from \cite{KKS2019}).
The assumption $\ln(n)=\oh{k}$ can be weakened to at least $n = \oh{2^k /k}$ at the expense of more complicated calculations.
Similarly, the assumption on superlinear growth of $n$ can be replaced by $n > \alpha k$ for some constant $\alpha$.
It could be dropped completely if the edges of $H_k(n,m)$ were chosen with replacement.

\subsection{Algorithms for two coloring.}

\begin{theorem}
\label{thm:constructive}
	For any fixed $\alpha< 1/2$, superlinear polynomially bounded function $n=n(k)$ and
	\[
			m \leq 2 \alpha \frac{n \ln(k)}{k} \varphi(n,k)^{-1} 2^{k-1},
	\]
	there exists an efficient algorithm that a.a.s. finds a proper two coloring of random $k$-graph $H_k(n,m)$.
\end{theorem}

When the constructive aspects are considered, efficient usually means that the running time is bounded by a polynomial of the input size.
It is rather obvious, that the simple procedure that we analyze satisfy this requirement, as 
it can be easily seen to work in time $\Oh{n k m}$.
Moreover, since it depends only on a small fraction of the actual edges, for some carefully tailored (but hardly natural) representation of the hypergraph, for most inputs its running time can be even sublinear.

As far as we know, the problem of algorithmic coloring of $k$-graphs, when $n$ is polynomially bounded in $k$, has not been directly addressed before.
However some results, obtained for fixed $k$,  can be adapted.
Achlioptas, Kim, Krivelevich and Tetali presented in \cite{AKKT02} a procedure that finds proper colorings for $H_k(n,m)$ when the number of edges $m$ is bounded by some function that is $O(n/k\cdot  2^k)$.
They also argued that their analysis of the procedure is tight up to the constants.
Literally, their proof requires $n = \omega( k^2 )$, but it can be easily modified to cover also the case when $n= \Th{k^2}$, with the cost of constants only.
The procedure constructs an equitable coloring by coloring consecutive pairs of vertices into distinct colors.
The vertices are chosen according to some fixed order as long as there are no edges which are monochromatic so far but have at most three not colored vertices.
In such a case it picks two not colored vertices of such an edge.
For a properly bounded number of edges in a random $k$-graph, a.a.s. the procedure never observes a monochromatic edge with only one not colored vertex, and hence succeeds in finding a proper coloring.
Within the same framework of fixed $k$, Achlioptas and Coja-Oghlan in \cite{ACO08} observed that the shape of the space of proper colorings of $H_k(n, r n)$ suddenly changes when $r$  passes value $(1 + \epsilon_k )2^{k-1}\ln(k)/k$.
That value was called the \emph{shattering threshold}.
One of the consequences of the change is that the space becomes disconnected (two colorings are considered adjacent if they are of small Hamming distance).
Although, above that threshold and until $m = \Th{n 2^k}$, proper colorings almost surely exist, the problem of finding one becomes significantly harder. 
When $n$ is polynomially bounded in $k$, we observe somewhat analogous behavior for the number of edges of the order $\Th{n \cdot \ln(k)/k \cdot 2^k}$.
The analysis of our procedure can be extended to justify that, for $\alpha < 1/2$, a proper coloring is usually surrounded by a large number of other proper colorings that are of small Hamming distance.
On the other hand, similar arguments can be used to justify that, when $\alpha = \omega(1)$, proper colorings tend to be isolated.

\section{Proof of Theorem \ref{thm:main}}

We fix $\varepsilon>0$ and assume that $n=n(k)$ and $c=c(n,k)$ satisfy the assumptions of Theorem \ref{thm:main}.
%altering the probabilistic space
In the canonical definition of $H_k(n,m)$, edges of the hypergraph are sampled without replacement.
That causes certain purely technical complications in the calculations.
In order to avoid them we alter the probabilistic space slightly.
Let $H'_k(n,m)$ be a random multi-hypergraph in which $m$ edges are sampled with replacement from the set of all $k$-subsets of a set of size $n$.
Clearly $H'_k(n,m)$ conditioned on the event that all edges are distinct is $H_k(n,m)$.
However, for our parameters that is asymptotically almost sure event.
Indeed, classical analysis of the birthday paradox problem implies that for $m = \oh{\binom{n}{k}^{1/2}}$ the probability of observing two identical edges is $\oh{1}$.
In our case, $m$ is exponential in $k$ (as $n= o(2^k)$ and $\varphi(n,k)^{-1} \leq 2^k$), while $\binom{n}{k}^{1/2} \geq \left(\frac{n}{k}\right)^{k/2}$ is superexponential (since $n$ is superlinear in $k$).
It means that within our parameters any asymptotically almost sure event in $H'_k(n,m)$ is also a.a.s. in $H_k(n,m)$.
From now on we work with $H'_k(n,m)$.

Let $H$ be a random hypergraph sampled with distribution $H'_k(n,m)$,
$X$ denote the number of proper equitable (i.e. every color class consists of exactly $n/2$ vertices) two colorings of $H$, and $X_2$ be the number of ordered pairs of such colorings. By Payley-Zygmund inequality we get
\[
	\prob[X>0] \geq \frac{\expt[X]^2}{\expt[X_2]}.
\]
We aim at showing that, for large $k$, hypergraph $H$ is two colorable with high probability so we would like to have
\[
	\frac{\expt[X]^2}{\expt[X_2]} \sim 1,
\]
for $k$ tending to infinity. % and reasonably bounded $n$ (like e.g. $n= \Theta(k^2)$ ???).
Since the above quotient is bounded from above by 1 it is enough to prove that 
\begin{equation}
\label{eq:mainQuotient}
	\frac{\expt[X_2]}{\expt[X]^2} \leq  1+\oh{1}.
\end{equation}
We have
\[
	\expt[X]= {n \choose n/2} \left(1-2 \frac{{n/2 \choose k}}{{n \choose k}}\right)^m
\]
and
\[
	\expt[X_2]= \sum_{a=0}^{n/2} {n \choose n/2} {n/2 \choose a}^2 \left( 1- 4 \frac{{n/2 \choose k}}{{n \choose k}}
			+2 \frac{{a \choose k}}{{n \choose k}}+2 \frac{{n/2-a \choose k}}{{n\choose k}}\right)^m,
\]
where $a$ corresponds to the size of the intersection of the sets of red vertices in two equitable colorings,
binomial coefficients stand for the number of pairs of equitable colorings at Hamming distance $2a$,
and the expression under the $m$-th power is the inclusion-exclusion formula for the probability that a random edge is properly colored by both equitable colorings of a fixed pair with Hamming distance $2a$.
Denote
\[
	F(a)= \frac{{n/2 \choose a}^2}{{n \choose n/2}}
	\;\;\;\text{and}\;\;\;
	G(a)=  \left( \frac{1- 4 \frac{(n/2)^{\underline{k}}}{n^{\underline{k}}}
				+2 \frac{a^{\underline{k}}}{n^{\underline{k}}}
				+2 \frac{(n/2-a)^{\underline{k}}}{n^{\underline{k}}}}
				{\left(1-2 \frac{(n/2)^{\underline{k}}}{n^{\underline{k}}}\right)^2} \right)^m,
\]
so that
\begin{equation}
\label{eq:mainSum}
\frac{\expt[X_2]}{\expt[X]^2} = \sum_{a=0}^{n/2} F(a) \cdot G(a).
\end{equation}
It is natural idea to apply the Laplace method for the above sum -- i.e. determine the range of $a$ that constitute \emph{central interval}, prove that properly rescaled sum within that interval is close to a Gaussian integral and finally show that contributions of the remaining ranges of $a$ are negligible.
We generally follow this path. We manage however to avoid a lot of technical details.

We choose $\delta = n \ln(n) / k$ and define the \emph{central interval} as
$\mathcal{C} = (\delta,n/2-\delta)\cap \Nat$ and \emph{tail intervals} as $\tl1=[0,\delta]\cap \Nat$ and $\tl2=[n/2-\delta, n/2]\cap \Nat$.
Note that our assumption on $n$ implies that $\delta/n= \oh{1}$, in particular the central interval is not empty.
In the next section we show that the contribution from the central interval to the sum (\ref{eq:mainSum}) is at most $1+\oh{1}$.
Then we show that the contributions from the tails are negligible.
That completes the proof of inequality (\ref{eq:mainQuotient}) and Theorem \ref{thm:main}.

\subsection{Central region}
We start with showing that $G(a)$ is essentially constant in $\mathcal{C}$.
\begin{prop}
We have
\[
	G(a) \leq 1+ \oh{1},
\]
uniformly for $a\in \mathcal{C}$.
\end{prop}
\begin{proof}
Function $G(a)$ is convex and symmetric on the central interval.
Therefore is has to be maximized in the ends of $\mathcal{C}$.
We assume that $k$ is large enough that $n/2 - \delta > k$.
For any $a\in \mathcal{C}$, we get
\begin{align*}
G(a) \leq G(\delta) & \leq \left( \frac{1- 2 \frac{\varphi}{2^{k-1}}
				+4 \frac{(n/2 -\delta)^{\underline{k}}}{n^{\underline{k}}}}
				{\left(1-\frac{\varphi}{2^{k-1}}\right)^2} \right)^m
\\
& \leq \left( \frac{1- 2 \frac{\varphi}{2^{k-1}}
				+4 \frac{(n/2)^{\underline{k}}}{n^{\underline{k}}}\frac{(n/2 -\delta)^{k}}{(n/2)^{k}}}
				{\left(1-\frac{\varphi}{2^{k-1}}\right)^2} \right)^m				
\\
& \leq \left(1+ \frac{
		\varphi \cdot 2^{-(k-2)} \cdot (1 - 2 \delta/n)^{k}
	}	
	{\left(1-\frac{\varphi}{2^{k-1}}\right)^2}\right)^m	 \\
& \leq \exp \left( m \cdot \varphi \cdot 2^{-(k-2)} \cdot (1 - 2 \delta/n)^{k} \cdot \left(1-\frac{\varphi}{2^{k-1}}\right)^{-2}\right).		
\end{align*}
We focus on the exponent with the intention of showing that it is $\oh{1}$.
We have
\begin{align*}
m \cdot \varphi \cdot 2^{-(k-2)} \cdot (1 -2 \delta/n)^{k} \cdot \left(1-\frac{\varphi}{2^{k-1}}\right)^{-2}
=
\Oh{ n \cdot \exp(-2 \delta k /n) }.
\end{align*}
Since we chose $\delta = n \ln(n)/k$ the above expression is $\oh{1}$, hence $G(a)\leq G(\delta) = 1+\oh{1}$.
\end{proof}

By the above Proposition we have that
\[
	\sum_{a\in \mathcal{C}} F(a)\cdot G(a) \leq (1+\oh{1})\cdot  \sum_{a\in \mathcal{C}} F(a) < (1+\oh{1}) \cdot \sum_{a=0}^{n/2} \frac{{n/2 \choose a}^2}{{n \choose n/2}}.
\]
The last sum, however is well known to be equal to 1. Altogether we get that
\[
	\sum_{a\in \mathcal{C}} F(a)\cdot G(a) \leq 1+\oh{1}.
\]

\subsection{Tails}
By symmetry of the analyzed functions it is sufficient to consider only one of the tails.
We focus on $\tl1$.
Functions $F(.)$ and $G(.)$ are respectively increasing and decreasing on $\tl1$.
Therefore for any $a\in \tl1$, we have
\[
	F(a) \cdot G(a) \leq F(\delta)\cdot G(0).
\]
Now
\begin{align*}
	G(0) = \left( 1- 2 \frac{(n/2)^{\underline{k}}}{n^{\underline{k}}} \right)^{-m}
	= \left( 1- \frac{\varphi}{2^{k-1}} \right)^{-m}
	= \exp \left(\frac{m \varphi}{2^{k-1}} + \Oh{m 2^{-2k}\varphi^2}\right)
	\\
	= (1 + \oh{1}) \cdot \exp( c \cdot \varphi \cdot n).
\end{align*}
Note that we used the fact that $n= \oh{2^k}$ in order to obtain $m 2^{-2k}\varphi^2= \oh{1}$.

In order to bound $F(\delta)$ we use the well known facts that ${n \choose n/2} > \frac{2^n}{n}$ and ${n/2 \choose \delta} < 2^{(n/2) H(2 \delta/n)}$, where $H(.)$ is the binary entropy function.
We get
\begin{align*}
	F(\delta) < n \cdot \left( 2^{H(2 \delta/n) -1} \right)^n.
\end{align*}
Combining both bounds we obtain
\[
	F(a)\cdot G(a) \leq (1+\oh{1})\cdot n \cdot \left( \exp(c \cdot \varphi - \ln(2) (1-H(2 \delta/n)) \right)^n.
\]
Since $\delta/n = \oh{1}$ we also have $H(2 \delta/n)=\oh{1}$.
Therefore, whenever $c \cdot \varphi < (1-\epsilon) \ln(2)$ the above expression is exponentially small in $n$.
In particular (for sufficiently large $n$) it is still $\oh{1}$ when summed over $\tl1$.

\section{Algorithmic coloring of random $k$-graphs}
Theorem \ref{thm:constructive} has been stated in terms of uniform random hypergraph $H_k(n,m)$.
However, closely related binomial model allows for much more natural proof.
In a random hypergraph $H_k(n,p)$ every $k$-subset of the set of vertices of size $n$ is independently added to the hypergraph with probability $p$.
We are going to work with the following technical statement.

\begin{prop}
\label{prop:constructiveP}
	For any fixed $\alpha< 1/2$, and any superlinear polynomially bounded function $n=n(k)$, there exists an efficient algorithm that on random $k$-graph $H_k(n,p)$ with the expected number of edges
	\[
		m \leq 2 \alpha \frac{n \ln(k)}{k} \varphi(n,k)^{-1} 2^{k-1}
%		\alpha k \ln(k) \cdot 2^{k-1}
	\]
	a.a.s. finds a proper two coloring.
\end{prop}

The number of edges in $H_k(n,p)$ is distributed binomially, and hence it is concentrated.
Together with the fact that the property of being two colorable is monotonic, it implies that the statement analogous to the above proposition is also true for the uniform model $H_k(n,m)$.
Hence the above proposition implies Theorem \ref{thm:constructive}.
For the rest of this section we work with the binomial model.

Let $\alpha< 1/2$ and $n=n(k)$ satisfy the assumptions of Proposition \ref{prop:constructiveP}. 
We define $p$ and $q$ as
\[
	p \cdot {n \choose k} = \varphi^{-1} q 2^{k-1}, \;\;\;\;\;\; q = 2 \alpha \frac{n \ln(k)}{k},
\]
so that the expected number of edges is $m$ as in Proposition \ref{prop:constructiveP}.
Then, the expected number of monochromatic edges in an equitable coloring is
\[
	2  \cdot {n/2 \choose k} \cdot p =  q.
\]

For a given coloring, an edge is \emph{almost monochromatic} if all but one of its vertices are of the same color.
The vertex of that edge with the unique color is called its \emph{head}.
A vertex is called \emph{safe} if it is not the head of any almost monochromatic edge.

\subsection{Procedure}

Let $\redset_0$, $\blueset_0$ be a fixed equipartition of the vertex set of $H=H_k(n,p)$.
We initially color the vertices of $\mathcal{R}_0$ red and the vertices of $\mathcal{B}_0$ blue.
Since, in our case, the input hypergraph is random we can work with a fixed partition.
If the algorithm were to be applied to a specific instance, it seems more appropriate to start with a random one.
The general idea of the procedure is to keep recoloring safe vertices belonging to monochromatic edges until the coloring becomes proper.
Note that changing the color of a safe vertex does not create new monochromatic edges.
Additionally, for convenience, we want to keep the coloring equitable.
Therefore, in every round, we try to repair two monochromatic edges of different colors by switching the colors of two safe vertices.

The main loop of the algorithm is presented as Listing \ref{alg:main}.
In a single step, it switches colors of two carefully chosen safe vertices.
Procedures \FuncSty{PickRedVertex()} and  \FuncSty{PickBlueVertex()} are responsible for the choices.
The first of them is presented on Listing \ref{alg:pickRed}. The other one is symmetric.
Global variables $\redset$ and $\blueset$ denote the current partition of the vertices into red and blue.
Initially, $\redset = \redset_0$ and $\blueset = \blueset_0$.
These variables are modified only in the main loop.
They are also used implicitly in the safety checks within choice procedures.
Values $N_R, N_B$ are the numbers of initially red and blue monochromatic edges, respectively.
The value of $r$ (resp. $b$) corresponds to the index of currently or most recently considered initially red (resp. blue) edge.
Both these variables are global.
Beside the main loop, they are used in the corresponding choice procedures.

When the algorithm reaches the point when all initially monochromatic edges are already considered (i.e. $r>N_R$ and $b>N_B$), the main loop ends
and the current coloring (that is proper for the input hypergraph) is returned.
However the algorithm may fail if, at some point, function \FuncSty{PickRedVertex()} or  \FuncSty{PickBlueVertex()} is unable to find a good vertex to be recolored.

\begin{algorithm}[H]
\label{alg:main}
$\redset \gets \redset_0$,
$\blueset \gets \blueset_0$,
$r\gets 1$,
$b\gets 1$ \;
 \While{$r \leq N_R $ or $b \leq N_B$}{
%    \tcc{find a safe red vertex to be recolored}
 	$v_R \gets$ \FuncSty{PickRedVertex()} \;
% 	\tcc{find a safe blue vertex to be recolored}
 	$v_B \gets$ \FuncSty{PickBlueVertex()} \;
    \tcc{switch colors of vertices}
    $\redset \gets \redset \setminus \{v_R\} \cup \{v_B\}$ \;
    $\blueset \gets \redset \setminus \{v_B\} \cup \{v_R\}$ \;
  } % EndWhile
  \Return{$(\redset, \blueset)$}
 \caption{The main procedure}
\end{algorithm}

We describe only procedure \FuncSty{PickRedVertex()}, the other one is symmetric.
Recall that $r$ is a global variable so in each consecutive call of the procedure it starts with the value of $r$ as left by the previous call.
Originally empty list $C_R$ contains the vertices that has been already tested for safety.
It is used only in \FuncSty{PickRedVertex()}, but we also need its state to be preserved between calls.

The procedure first advances into the next red edge and then iterates over its vertices (revealed in a random order) in search for a safe vertex or a vertex that has been recolored before.
It may happen that during the evaluation we run out of red edges (i.e. $r$ exceeds $N_R$) but we still need to do some recolorings to get rid of blue ones.
In order to deal with this issue we extend the list of red edges -- for any $r > N_R$, the $r$-th red edge is just uniformly chosen $k$-subset of $\redset_0$.
In fact we want the function to pick a randomly chosen red safe vertex -- extending the list of red edges with artificial random ones is just a way of realizing that task.

When iterating over the vertices of the current red edge, the procedure performs a few tests.
First it checks if the current vertex has been observed before (i.e. whether it belongs to $C_R$).
If not, it is appended to $C_R$ and checked for safety.
If the safety check succeeds, the vertex is returned from the function (and hence is picked as the one to be recolored next).
Otherwise, we continue iterating over the vertices of the current edge.

If the current vertex has been observed before, it was either recolored or unsafe at that time.
If it was recolored, then the current edge is no longer red and we can advance to considering the next edge.
This is done by restarting the function (the first instruction is to increase $r$ which points to the current red edge).
If the vertex was unsafe at the time of being checked, we treat it as still unsafe and continue iterating over the vertices of the current edge
(we ignore the possibility that the vertex might be safe in the current coloring).

It may happen that we reach the end of the current edge without finding safe or recolored vertex.
We distinguish two cases here based on the value of $r$.
If $r\leq N_R$, the problematic edge is an actual edge of the hypergraph.
In such a case we declare failure.
On the other hand, for $r > N_R$, the problematic edge does not belong to the hypergraph, and has been added artificially to the list.
In that case, we can simply ignore the problem and keep looking for another vertex to be recolored, by considering next artificial edges.

\begin{algorithm}[H]
\label{alg:pickRed}
	$r \gets r+1$     \tcp*[r]{advance to the next edge}
	\For{$j=1, \ldots, k$}{
		$v \gets $ reveal the $j$-th vertex of the $r$-th red edge \;
		 \If {$C_R$ does not contain $v$}
        {
            append $v$ to $C_R$\; %$C_R \gets C_R \cup \{v\}$ \;
            \lIf{ $v$ is safe in the current coloring} {\Return $v$ }
        }
        \Else{
        	\If {$v$ has been recolored}{
        		restart the procedure
        	}
        }
	}
\lIf{ $r\leq N_R$ {\bf or} $C_R$ contains all elements of $\redset_0$} {FAIL}
	\lElse{ restart the procedure} 
 \caption{Procedure \FuncSty{PickRedVertex()}}
\end{algorithm}

\subsection{Analysis}

Random hypergraph $H_k(n,p)$ can be seen as the product of ${n \choose k}$ independent Bernoulli trials, each with probability of success $p$.
With every such trial, and hence with every $k$-subset $e$ of V, we associate the indicator random variable $I_e$ that corresponds the outcome of the trial.
In order to keep a clear distinction between potential edges and the actual edges of the hypergraph, instead of referring to a potential edge $f\subset V$ we refer to its indicator random variable $I_f$.

The probability distribution $H_k(n,p)$ over the $k$-graphs on $n$ vertices can be also obtained by other constructions.
We describe one that can be easily modified to suit our needs.
Let $N_R$ be a random variable with binomial distribution $B( {n/2 \choose k} , p)$.
We first determine the value of $N_R$, then chose uniformly that number of distinct $k$-subsets of $\redset_0$ and put them in the constructed hypergraph as edges.
The edges contained in $\blueset_0$ are constructed in the same way.
The remaining edges, i.e. the ones that are not monochromatic in the initial coloring, are added to the hypergraph just like before -- according to the values of the corresponding indicators.
So far the resulting probability distribution % on the set of all $k$-graphs over $n$ vertices
is just the same as in the original construction.
It turns out to cause technical difficulties that monochromatic edges chosen in such a way are constrained to be distinct and hence they are dependent.
For small enough $p$ (like the ones that we are interested in), the dependence is very weak.
To get rid of this constraint we alter the probabilistic space slightly.
Random multi-hypergraph $H'_k(n,p)$ is defined by an analogous construction with the only change that the $N_R$ (resp. $N_B$) edges contained in $\redset_0$ (resp. $\blueset_0$) are chosen with repetitions, and hence independently.
Clearly, $H'_k(n,p)$ conditioned on the event that no edge has been included more than once, is just $H_k(n,p)$.
That event happens with asymptotic probability $1$.
Indeed, birthday paradox problem tells that we need to pick roughly ${n/2 \choose k}^{1/2}$ independent red edges to observe a repeated one with positive probability.
On the other hand $N_R$, which has mean $q/2$, almost surely takes much smaller value.
This discussion shows, that if some property holds a.a.s. in $H'_k(n,p)$ it also has to hold a.a.s. in $H_k(n,p)$.
From now on we work with $H'_k(n,p)$.

The algorithm ends successfully when there are no more monochromatic edges, but it may fail earlier when no recoloring candidate is found.
We prove below that the algorithm asymptotically almost surely succeeds.
We concentrate on the red edges, as the situation of the blue ones is symmetric.
During the safety check, we reveal only the values of indicators $I_e$ for $k$-sets $e$, for which the checked vertex is the head, i.e. check if our $k$-graph actually contains corresponding sets as edges.
The initially monochromatic edges are revealed in a random order.
We deliberately altered the probabilistic space so that the consecutive red edges are drawn independently and uniformly from all $k$-subsets of $\redset_0$.
\begin{observation}
\label{obs:recol-steps}
	The initial number of red (resp. blue) edges a.a.s. does not exceed $q$.
\end{observation}
\begin{proof}
The number of initially red edges is distributed binomially with mean $q/2$.
Chernoff bound shows that probability of having more than $q$ red edges is exponentially small in $q$.
The same bound holds for the number of initially blue edges.
\end{proof}

\begin{observation}
\label{obs:safety-chance}
Safety checks on different red (resp. blue) vertices are performed on disjoint sets of indicators.
\end{observation}
\begin{proof}
Suppose to the contrary that the indicator of a set $e$ is checked for two red vertices $v$ and $w$.
Without loss of generality the check on $w$ is the later one.
But then during $v$'s check, $w$ is still red, so $e$ would contain at least two red vertices during the check, which is impossible.
\end{proof}

The last observation guarantees that each indicator is tested at most once during safety checks of red vertices.
In fact, most of the indicators are tested at most once during the whole evaluation of the algorithm.
However, a small number of indicators may be tested twice, for both colors.
Consider such an indicator $I_f$ of some set $f$, first tested during a safety check of a vertex $w$, and then again for a vertex $u$ of opposite color. 
Assuming without loss of generality that $w$ was red during the check, and $u$ was blue, the only way for this to happen is that $f \setminus\{w\}$ was first entirely blue, and then all the vertices of $f$ except $u$ and $w$ were recolored to red. 

Observe that it can only harm the algorithm if $f$ is indeed an edge, and thus may render $u$ or $w$ unsafe. We will call such situation -- a second safety check on an actual edge -- a \emph{corrupted safety check}.
But as $f$ must have been initially an almost-monochromatic edge, and then it would have to get almost all its vertices recolored, it is a very unlikely situation, and we will prove in Observation \ref{obs:no-corrupted} that a.a.s. no corrupted safety check will happen. 

For the sake on analysis, we slightly modify the safety checks to only truly inspect the indicators that were not tested before. For any indicator previously seen, we assume that the check automatically succeeds -- note that with no corrupted safety checks the assumption must be true. 

\begin{observation}
\label{obs:safety-check-prob}
Independently of the outcomes of the previous checks, the probability of a vertex passing the modified safety check is at least $\delta = k^{-2\alpha(1+\oh{1})}$.
\end{observation}
\begin{proof}
We compute the probability of a successful check on an arbitrary vertex $v$. It fails if there exists an almost monochromatic edge for which $v$ is the head.
There are at most $\binom{n/2}{k-1}$ candidates for such an edge, so the probability that none of them shows up is at least $(1-p)^{\binom{n/2}{k-1}}$.
We have
\begin{align*}
	p \binom{n/2}{k-1} = q \varphi^{-1} 2^{k-1} \frac{\binom{n/2}{k-1}}{\binom{n}{k}}
	= q \frac{k}{2 (n/2 - k+1)} \varphi^{-1} 2^{k} \frac{(n/2)^{\underline{k}}}{ n^{\underline{k}}} \\
	= q \frac{k}{(n - 2k+2)} = 2 \alpha \ln(k) \frac{n}{(n - 2k+2)} = 2\alpha \ln(k) (1+\oh{1})).
\end{align*}
Using the fact that $1-p = \expb{-p + \Oh{p^2}} = \expb{-p \cdot (1+\oh{1})}$ we obtain:
\begin{align*}
(1-p)^{\binom{n/2}{k-1}} = \exp\left[-p \cdot \binom{n/2}{k-1} (1+\oh{1}) \right] = k^{-2 \alpha (1+\oh{1})}.
\end{align*}
\end{proof}
We show next that, during the whole algorithm, on the list of checked vertices $C_R$, there is a significant fraction of the ones which have passed the safety check (and thus have been recolored).

\begin{observation}
\label{obs:checked-colored}
	Let $\alpha' = (1/2+\alpha)/2$.
	Asymptotically almost surely, for every $l$ such that $k/2 \leq l \leq n/2$, the prefix of $C_R$ of length $l$ contains at least $l \cdot  k^{-2\alpha'}$ recolored vertices.
\end{observation}
\begin{proof}
Let us first consider a chance for a long series on consecutive failed safety checks -- namely, $\frac{1}{2} k^{2 \alpha'}$ failures in a row.
Observe that such a sequence can only happen on indicators that were not previously checked (otherwise they would automatically suceed). By Observation \ref{obs:safety-chance} they are conducted on disjoint sets of vertices and thus are independent. The chance that we observe such a series of failures is, consequently, $(1-\delta)^{\frac{1}{2} k^{2 \alpha'}} < \expb{-\delta \cdot \frac{1}{2} \cdot k^{2\alpha'}} = \expb{-\frac{1}{2} k^{2(\alpha' - \alpha) +\oh{1}}}$.
The list $C_R$ reaches the length at most $n/2$, hence straightforward union bound shows that the probability of observing a long series of failures is smaller than $(n/2) \expb{-\frac{1}{2}k^{2(\alpha' - \alpha) +\oh{1}) }}= \oh{1}$.

This implies that every prefix of $C_R$ of length $l$ contains at least $\lfloor l \cdot 2 \cdot k^{-2 \alpha'} \rfloor$ recolored vertices.
However, since $l \geq k/2$ we get $l = \omega( k^{2 \alpha'} )$, which is sufficient to deduce that, for large enough $k$, we have 
$l \cdot k^{-2 \alpha'} \leq \lfloor l \cdot 2 \cdot k^{-2 \alpha'} \rfloor$.
\end{proof}

\begin{observation}
The modified procedure asymptotically almost never declares failure.
\end{observation}
\begin{proof}
The whole procedure fails if one of the choice procedures is unable to find a vertex to be recolored.
As usual we focus on \FuncSty{PickRedVertex} and show that such situation a.a.s. does not happen.

The first possibility for the failure is that the procedure does not find a safe vertex or recolored vertex in an edge that was initially red.
When we reveal the vertices of such an edge $e$, some of them are already checked. 
The procedure fails if all unchecked vertices fail the safety check, and no vertex of $e$ from $C_R$ was recolored (i.e. safe) when met for the first time. 
At the time when we start to analyze edge $e$, at least one of the following two conditions must hold: either at least $k/2$ vertices of $e$ are in $C_R$, or at least $k/2$ are yet unchecked. 

Suppose first that at least $k/2$ of the vertices of $e$ have been checked and added to $C_R$ (recall that each vertex is checked for safety only once, and if it fails it is treated as permanently unsafe.)
In the altered probabilistic space, the choices of the vertices of $e$ are not influenced by other monochromatic edges, so the set of $k/2$ vertices that belong to $C_R$ is chosen uniformly from $k/2$-subsets of $C_R$.
Due to Observation \ref{obs:checked-colored} the chance that they are all unsafe is smaller than $(1-k^{-2\alpha'})^{k/2}$.

The other case is that there are $k/2$ unchecked vertices. For the algorithm to fail on $e$, all of them must not pass the safety check.
Observation \ref{obs:safety-chance} implies that the chance of this event is at most $(1-\delta)^{k/2} <  (1-k^{-2 \alpha'})^{k/2}$. 
Therefore in both cases the chance that the procedure fails while considering edge $e$ is never bigger than $(1-k^{-2 \alpha'})^{k/2} < \expb{-\frac{1}{2}k^{1-2 \alpha'}}$. 

By union bound and Observation \ref{obs:recol-steps}, the procedure fails explicitly with probability at most $2 n/k \cdot \ln k \cdot \expb{-\frac{1}{2}k^{1-2\alpha'}} = \oh{1}$.

Note that once $r$ exceeds $N_R$ a failure can only occur when all the vertices of $\redset_0$ are already checked.
Hence the necessary condition for the procedure to fail at that point is that there are no safe unchecked vertices at the time when the procedure is called. 
We show that the probability of that event is very small.
By the construction of the main procedure, at each time when \FuncSty{PickRedVertex} is called, the number of recolored red vertices does not exceed $\max(N_R, N_B)$.
Both these values are smaller than $q$ a.a.s.
It means that, at each time when the procedure is called, list $C_R$ contains at most $q$ recolored vertices.
Observation \ref{obs:checked-colored} implies that it contains at most $(q+1)k^{1-2 \alpha'}= \oh{n}$ vertices. 
Therefore $(n/2)(1-\oh{1})$ red vertices are still not checked.
The probability that they all turn out to be unsafe is 
$
	(1-\delta)^{(n/2)(1+\oh{1})} < \exp (- \delta (n/2) (1+\oh{1})  )
$
which is $\oh{\exp(-q)}$.
Since the number of calls to the procedure is a.a.s. smaller than $q$, the situation which allows the procedure to fail because all the vertices of $\redset_0$ are already checked a.a.s. does not happen.
\end{proof}

\begin{observation}
	Asymptotically almost surely, no corrupted safety checks will happen.
	\label{obs:no-corrupted}
\end{observation}

\begin{proof}
	Every time, when a blue vertex fails the safety check, it is the head of at least one edge that is currently almost red.
	The tail of every such edge is called a \emph{witnessing tail}.
	It is crucial to observe that it is completely irrelevant for the further evaluation of the modified procedure what is the specific content of the witnessing tails.
	Moreover, it is not even necessary to reveal that content to the coloring procedure.
	It means that, conditioned on the specific outcomes of the safety checks, the witnessing tails are distributed uniformly among the $(k-1)$-subsets of the sets of vertices that were of the appropriate color at the time of witnessing.
	
	A corrupted safety check happens when some witnessing tail $f$ subsequently gets $k-2$ of its vertices recolored. For the sake of analysis we postpone the revealing of the content of the witnessing tails till the end of the procedure, and then see if some of them caused a corrupted check.
	
	In the final coloring we have at most $q$ vertices recolored to red.
	For a given witnessing tail, the probability that $k-2$ of its vertices are contained in the set of the recolored vertices is at most
	\[
		\frac{\binom{q}{k-2} \cdot(n/2)}{ \binom{n/2}{k-1}} 
		\leq 
		n \cdot \frac{q^{\underline{k-1}}}{(n/2)^{\underline{k-1}}}
		\leq
		n \cdot \frac{q^{k-1}}{(n/2)^{k-1}}
		= n \cdot (2q/n)^{k-1}.
	\]
	 
	The expected number of witnessing tails for every vertex is 
	$p \cdot \binom{n/2}{k-1} = 2\alpha \ln(k) (1+\oh{1}))$ 
	(derived already in Observation \ref{obs:safety-check-prob}). 
	Therefore the expected final number of witnessing tails is $\Oh{n \ln(k)}$, and by Markov equality, asymptotically almost surely there are no more than $n^2$ witnessing tails. 
	But then, by simple union bound, the expected number of corrupted safety checks does not exceed $n^2 \cdot n \cdot (2q/n)^{k-1} = \oh{1}$. 
	Hence, a.a.s. not even one of them happens.
\end{proof}

In many places in the presented proofs we used quite rough estimations.
It is interesting to observe that we did not lose much.
Observation \ref{obs:safety-chance} shows that for $\alpha>1/2$, the fraction of safe vertices among $\redset_0$ is $\oh{k^{-1}}$. 
Therefore a random red edge with high probability avoids them all.
In consequence, in these cases, our procedure is very likely to fail.

\bibliographystyle{siam}
\bibliography{pb}

\end{document}